\documentclass[12pt,oneside]{amsart}

\usepackage[a4paper]{geometry}
\usepackage[english]{babel}
\usepackage[latin1]{inputenc}
\usepackage[T1]{fontenc}
\usepackage{lmodern}
\usepackage[all]{xy}
\usepackage{amsfonts,amssymb}
\usepackage{amsmath}
\usepackage{enumerate}
\usepackage{breqn}
\usepackage{tikz}
\usepackage{pgfplots}
\usepackage[numbers,sort]{natbib}
\usepackage{bbm}

\newcommand{\e}{{\,{\rm e}}}

\newcommand{\PP}{\ensuremath{\mathbb{P}}}

\newcommand{\EE}{\ensuremath{\mathbb{E}}}

\newcommand{\I}{\ensuremath{\mathbf{I}}}
\newcommand{\J}{\ensuremath{\mathbf{J}}}
\newcommand{\cC}{\ensuremath{\mathcal{C}}}

\newcommand{\s}{\ensuremath{\omega}}

\newcommand{\val}{\ensuremath{\nu}}
\newcommand{\K}{\ensuremath{K}}
\renewcommand{\L}{\ensuremath{L}}

\newcommand{\R}{\ensuremath{\mathbb{R}}}

\newcommand{\bC}{\ensuremath{\mathbf{C}}}
\newcommand{\bU}{\ensuremath{\mathbf{U}}}
\newcommand{\bA}{\ensuremath{\mathbf{A}}}

\def\ind{{\mathbbm{1}}}
\newtheorem{lemma}{Lemma}
\newtheorem{theorem}{Theorem}

\newtheorem{remark}[lemma]{Remark}

\begin{document}

\title[Sampling without replacement]{Weighted sampling without replacement}
\author{Anna Ben-Hamou, Yuval Peres and Justin Salez}
\address{}
\email{}

\keywords{}
\subjclass[2010]{}

\begin{abstract}
Comparing concentration properties of uniform sampling with and without replacement has a long history which can be traced back to the pioneer work of Hoeffding \cite{hoeffding}. The goal of this short note is to extend this comparison to the case of non-uniform weights, using a coupling between samples drawn with and without replacement. When the items' weights are arranged in the same order as their values, we show that the induced coupling for the cumulative values is a submartingale coupling. As a consequence, the powerful Chernoff-type upper-tail estimates known for sampling with replacement automatically transfer to the case of sampling without replacement. For general weights, we use the same coupling to establish a sub-Gaussian concentration inequality. As the sample size approaches the total number of items, the variance factor in this inequality displays the same kind of sharpening as Serfling \citep{serfling} identified in the case of uniform weights. We also construct an other martingale coupling which allows us to answer a question raised by Luh and Pippenger \cite{luh2014large} on sampling in Polya urns with different replacement numbers.
\end{abstract}

\maketitle

\section{Introduction}

In a celebrated paper \citep{hoeffding}, Hoeffding first singled out a fruitful comparison between sampling with and without replacement: any linear statistics induced by uniform sampling without replacement in a finite population is less, in the convex order, than the one induced by sampling with replacement. In particular, all the Chernoff-type tail estimates that apply to sampling with replacement (the sample then being i.i.d.) automatically apply to sampling without replacement. As the sample size increases, it is natural to expect that sampling without replacement should concentrate even more, in the sense that, when the sample size approaches the total number of items, the variance should not be of the order of the number of sampled items, but of the number of \emph{unsampled} items. This was verified by Serfling in \cite{serfling}.

One natural question is to determine whether a similar comparison also holds when the sampling procedure is no longer uniform and when different items have different weights. 

More precisely, consider a collection of $N$ items $1\leq i\leq N$, each equipped with a \emph{weight} $\s(i)>0$ and  a \emph{value} of interest $\val(i)\in\R$. We assume that
\begin{displaymath}
\sum_{i=1}^N \s(i)=1\, .
\end{displaymath}
Let $X$ be the cumulative value of a sample of length $n\leq N$ drawn without replacement and with probability proportional to weights, i.e. 
\begin{eqnarray*}X & := & \val({\I_1})+\cdots+\val({\I_n}),\end{eqnarray*} 
where for each $n-$tuple $(i_1,\ldots,i_n)$ of distinct indices in $\{1,\ldots,N\}$, 
\begin{eqnarray*}
\PP\left((\I_1,\ldots,\I_n)=(i_1,\ldots,i_n)\right) & = & \prod_{k=1}^n\frac{\s({i_k})}{1-\s({i_1})-\cdots-\s({i_{k-1}})}\, .
\end{eqnarray*} 

A much simpler statistic is the one that arises when the sample is drawn with replacement, namely
\begin{eqnarray*}Y& := & \val({\J_1})+\cdots+\val({\J_n}),\end{eqnarray*}  where now for each $n-$tuple $(j_1,\ldots,j_n)\in\{1,\ldots,N\}^n$,
\begin{eqnarray*}
\PP\left((\J_1,\ldots,\J_n)=(j_1,\ldots,j_n)\right) & = & \prod_{k=1}^n\s({j_k})\, .
\end{eqnarray*}

One particular case is when weights and values are arranged in the same order, \textit{i.e.}
\begin{eqnarray}
\label{assume}
\s(i)>\s(j) & \Longrightarrow & \val(i)\geq \val(j)\, .
\end{eqnarray}
\begin{theorem}
\label{th:main}
Assume that condition \eqref{assume} holds. Then $X$ is less than $Y$ in the \emph{increasing convex order}, i.e. for every non-decreasing, convex function $f\colon\R\to\R$,
\begin{eqnarray}
\label{icx}
\EE\left[f\left(X\right)\right] & \leq & \EE\left[f\left(Y\right)\right]\, .
\end{eqnarray}
\end{theorem}

Our second result is a sub-Gaussian concentration inequality for $X$ in the case of arbitrary weights $(\s(i))_{i=1}^N$. Define
\begin{displaymath}
\Delta:=\max_{1\leq i\leq N}\val(i)-\min_{1\leq i\leq N}\val(i)\qquad\text{and}\qquad\alpha=\frac{\min_{1\leq i\leq N}\s(i)}{\max_{1\leq i\leq N}\s(i)}\, .
\end{displaymath}
The case $\alpha=1$ (uniform sampling) was analysed by Serfling \citep{serfling}. 

\begin{theorem}
\label{thm:gaussian}
Assume $\alpha<1$. For all $t>0$,
\begin{displaymath}
\max\left\{\PP\left(X-\EE X >t\right),\PP\left(X-\EE X<-t\right)\right\}\leq \exp\left(-\frac{t^2}{2v}\right)\, ,
\end{displaymath}
with
\begin{eqnarray}\label{eq:def-variance}
v&=&
\min\left(4\Delta^2 n\, ,\, \frac{1+4\alpha}{\alpha(1-\alpha)}\Delta^2 N \left(\frac{N-n}{N}\right)^\alpha\right) 
\end{eqnarray}
\end{theorem}

We also answer a question raised by \cite{luh2014large}. The problem is to compare linear statistics induced by sampling in Polya urns with replacement number $d$ versus $D$, for positive integers $d,D$ with $D>d\geq 1$. 


Let $\bC$ be a population of $N$ items, labelled from $1$ to $N$, each item $i$ being equipped with some value $\val(i)$. Let $d<D$ be two positive integers. For $n\geq 1$, let $(\K_1,\dots,\K_n)$ and $(\L_1,\dots,\L_n)$ be samples generated by sampling in Polya urns with initial composition $\bC$ and replacement numbers $d$ and $D$ respectively, \textit{i.e.} each time an item is picked, it is replaced along with $d-1$ (resp. $D-1$) copies. We say that $(\K_1,\dots,\K_n)$ (resp. $(\L_1,\dots,\L_n)$) is a $d$-Polya (resp. $D$-Polya) sample. Let
\begin{eqnarray*}
W&=&\val(\K_1)+\dots +\val(\K_n)\, ,\\
Z&=&\val(\L_1)+\dots +\val(\L_n)\, .
\end{eqnarray*}

\begin{theorem}\label{thm:polya}
The variable $W$ is less than $Z$ in the \emph{convex order}, i.e. for every convex function $f\colon\R\to\R$,
\begin{eqnarray*}
\EE\left[f\left(W\right)\right] & \leq & \EE\left[f\left(Z\right)\right]\, .
\end{eqnarray*}
\end{theorem}

\begin{remark}
\cite{luh2014large} proved a similar result in the case where the first sample is drawn without replacement in $\bC$ and the second is a $D$-Polya sample, for $D\geq 1$.
\end{remark}

\section{Related work}

Weighted sampling without replacement, also known as \emph{successive sampling}, appears in a variety of contexts (see \cite{successive1, successive4, gordon, successive3}). When $n<<N$, it is natural to expect $Y$ to be a good approximation of $X$. For instance, the total-variation distance between $\PP\left(\I_{n+1}\in\,\cdot\;\Big| (\I_k)_{k=1}^n\right)$ and $\PP\left(\J_1\in\cdot\;\right)$ is given by $\displaystyle{\sum_{k=1}^n \s(\I_k)}$, which is $O(n/N)$ provided all the weights are $O(1/N)$.

Under the monotonicity assumption \eqref{assume}, Theorem \ref{th:main} establishes an exact strong stochastic ordering between $X$ and $Y$. Since $\J_1,\ldots,\J_n$ are independent copies of $\I_1$, the innumerable results on sums of independent and identically distributed random variables apply to $Y$. In particular, Chernoff's bound
\begin{eqnarray}
\label{bernstein}
\PP\left(Y\geq a\right) & \leq & \exp\left(n\Lambda(\theta)-\theta a\right)\, ,
\end{eqnarray}
yields a variety of sharp concentration results based on efficient controls on the log-Laplace transform $\Lambda(\theta)=\ln\EE[e^{\theta\val(\I_1)}]$. This includes the celebrated Hoeffding and Bernstein inequalities, see the book \cite{concentration}. Theorem \ref{th:main} implies in particular that all upper-tail estimates derived from Chernoff's bound (\ref{bernstein}) apply to $X$ without modification.

The condition \eqref{assume} describes a sampling procedure which is sometimes referred to as size-biased sampling without replacement. It arises in many situations, including ecology, oil discovery models, in the construction of the Poisson-Dirichlet distribution (\cite{permutation1,permutation2}), or in the configuration model of random graphs (\cite{configuration1,configuration2}).

Stochastic orders provide powerful tools to compare distributions of random variables and processes, and they have been used in various applications \cite{ordering2,ordering3,ordering}. As other stochastic relations, the increasing convex order is only concerned with marginal distributions. One way of establishing (\ref{icx}) is thus to carefully construct two random variables $X$ and $Y$ with the correct marginals on a common probability space, in such a way that 
\begin{eqnarray}
\label{submartingale}
X & \leq & \EE[Y|X]
\end{eqnarray}
holds almost-surely. The existence of such a \emph{submartingale coupling} clearly implies (\ref{icx}), thanks to Jensen's inequality. Quite remarkably, the converse is also true, as proved by Strassen \cite{strassen}. Similarly, Theorem \ref{thm:polya} is equivalent to the existence of a \emph{martingale coupling} $(W,Z)$.

\begin{remark}[The uniform case] When $\s$ is constant, \textit{i.e.} $\alpha=1$, the sequence $(\I_1,\ldots,\I_n)$ is exchangeable. In particular, $\EE[X]=\EE[Y]$,  forcing equality in (\ref{submartingale}). Thus, (\ref{icx}) automatically extends to arbitrary convex functions. This important special case was established five decades ago by Hoeffding in his seminal paper \cite{hoeffding}. Since then, improvements have been found as $n/N$ approaches $1$ \cite{serfling, bardenet}. Another remarkable feature of uniform sampling without replacement is the \emph{negative association} of the sequence $(\val(\I_1),\ldots,\val(\I_n))$ \cite{joag1983negative}. However, this result seems to make crucial use of the exchangeability of $(\I_1,\dots,\I_n)$, and it is not clear whether it can be extended to more general weights, \emph{e.g.} to monotone weights satisfying \eqref{assume}. Non-uniform sampling without replacement can be more delicate and induce counter-intuitive correlations, as highlighted by Alexander \cite{alexander1989counterexample}, who showed that for two fixed items, the indicators that each is in the sample can be positively correlated.
\end{remark}

Theorem \ref{thm:gaussian} holds under the only assumption that $\alpha<1$, but the domain of application that we have in mind is when $\alpha$ is bounded away from $0$ and $1$. In this domain, when $n\leq qN$, for some fixed $0<q<1$, equation \eqref{eq:def-variance} gives $v=O(\Delta^2n)$, which corresponds to the order of the variance factor in the classical Hoeffding inequality. When $n/N\underset{N\to\infty}\to 1$, then it can be improved up to $v=O\left(\Delta^2 n \left(\frac{N-n}{N}\right)^\alpha\right)$. In the uniform case $\alpha=1$, Serfling \citep{serfling} showed that $X$ satisfies a sub-Gaussian inequality with $v=\Delta^2 n\frac{N-n+1}{4N}$, implying that the variance factor has the order of the minimum between the number of sampled and unsampled items.

\subsection*{Organization.}
Both Theorems \ref{th:main} and \ref{thm:gaussian} rely on a coupling between samples drawn with and without replacement, which is constructed in Section \ref{sec:coupling}. Then, Theorems \ref{th:main}, \ref{thm:gaussian}, \ref{thm:polya} are proved respectively in Sections \ref{sec:proof-submartingale}, \ref{sec:proof-gaussian} and \ref{sec:proof-polya}.

\section{The coupling}
\label{sec:coupling}

The proofs of Theorem \ref{th:main} and \ref{thm:gaussian} rely on a particular coupling of samples drawn with and without replacement. This coupling is inspired by the one described in \cite{luh2014large} for the uniform case.

First generate an infinite sequence $(\J_k)_{k\geq 1}$ by sampling with replacement and with probability proportional to $(\s(i))_{i=1}^N$. Now, ``screen'' this sequence, starting at $\J_1$ as follows: for $1\leq k\leq N$, set
\begin{eqnarray*}
\I_k&=&\J_{T_k}\, ,
\end{eqnarray*}
where $T_k$ is the random time when the $k^{\text{th}}$ distinct item appears in $(\J_i)_{i\geq 1}$. 

The sequence $(\I_1,\dots,\I_n)$ is then distributed as a sample without replacement. As above, we define $\displaystyle{X=\sum_{k=1}^n \val(\I_k)}$ and $\displaystyle{Y=\sum_{k=1}^n \val(\J_k)}$.

\section{Proof of Theorem \ref{th:main}}
\label{sec:proof-submartingale}

Consider the coupling of $X$ and $Y$ described above (Section \ref{sec:coupling}). Under the monotonicity assumption \eqref{assume}, we show that $(X,Y)$ is a submartingale coupling in the sense of \eqref{submartingale}. As the sequence $(\J_1,\dots,\J_n)$ is exchangeable and as permuting $\J_i$ and $\J_j$ in this sequence does not affect $X$, it is sufficient to show that $\EE\left[\val(\J_1)\big| X\right]\geq X/n$.

Let $\{i_1,\dots_,i_n\}\subset\{1,\dots,N\}$ be a set of cardinality $n$, and let $\bA$ be the event $\{\I_1,\dots,\I_n\}=\{i_1,\dots_,i_n\}$.

\begin{displaymath}
\EE\left[\val(\J_1) \Big| \bA \right]= \sum_{j=1}^n \PP\left(\J_1=i_j \Big| \bA \right)\val(i_j) \, .
\end{displaymath}

Let us now show that, for all $1\leq k\neq \ell\leq n$, if $\val(i_k)\geq \val(i_\ell)$, then $\PP\left(\J_1 =i_k\Big| \bA \right)$ is not smaller than $\PP\left(\J_1 =i_\ell\Big| \bA \right)$. First, by \eqref{assume}, one has $\s(i_k)\geq \s(i_\ell)$. Letting $\mathfrak{S}_n$ be the set of permutations of $n$ elements, one has
\begin{eqnarray*}
\PP\left(\{\J_1=i_k\}\cap \mathbf{A}\right)
&=& \sum_{\pi\in \mathfrak{S}_n, \pi(1)=k} p(\pi)\, ,
\end{eqnarray*}
where
\begin{eqnarray*}
p(\pi)&:=&\s(i_{\pi(1)})\frac{\s(i_{\pi(2)})}{1-\s(i_{\pi(1)})} \cdots \frac{\s(i_{\pi(n)})}{1-\s(i_{\pi(1)})-\s(i_{\pi(2)})-\s(i_{\pi(n-1)})}
\end{eqnarray*}
Now, each permutation $\pi$ with $\pi(1)=k$ can be uniquely associated with a permutation $\pi^\star$ such that $\pi^\star(1)=\ell$, by performing the switch: $\pi^\star(\pi^{-1}(\ell))=k$, and letting $\pi(j)=\pi^\star(j)$, for all $j\not\in \{1,\pi^{-1}(\ell)\}$. Observe that $p(\pi)\geq p(\pi^\star)$. Thus
\begin{eqnarray*}
\PP\left(\J_1 =i_k\Big| \mathbf{A}\right)-\PP\left(\J_1 =i_\ell\Big| \mathbf{A}\right)&=&\frac{1}{\PP\left(\mathbf{A}\right)}\sum_{\pi\in \mathfrak{S}_n, \pi(1)=k}(p(\pi)-p(\pi^\star)) \geq 0\, .
\end{eqnarray*}
Consequently, by Chebyshev's sum inequality,
\begin{eqnarray*}
\EE\left[\val(\J_1) \Big| \mathbf{A}\right]&=&n\,\frac{1}{n}\sum_{j=1}^n \PP\left(\J_1=i_j \Big| \bA \right)\val(i_j)\\
&\geq & n\left(\frac{1}{n}\sum_{j=1}^n \PP\left(\J_1=i_j \Big| \bA \right)\right)\left(\frac{1}{n}\sum_{j=1}^n \val(i_j)\right)\\
&=& \frac{\sum_{j=1}^n \val(i_j)}{n}\, ,
\end{eqnarray*}
and $\EE\left[Y\Big| X\right]\geq X$.

\hfill\qedsymbol

\section{Proof of Theorem \ref{thm:gaussian}}
\label{sec:proof-gaussian}

We only need to show that the bound in Theorem \ref{thm:gaussian} holds for $\PP\left[X-\EE X >t\right]$. Indeed, replacing $X$ by $-X$ (i.e. changing all the values to their opposite) does not affect the proof. Hence, the bound on $\PP\left[X-\EE X <-t\right]$ will follow directly.

Theorem \ref{thm:gaussian} is proved using the same coupling between sampling with and without replacement as described in Section \ref{sec:coupling}. 

Note that, in this coupling, $X$ is a function of the \textsc{i.i.d.} variables $(\J_i)_{i\geq 1}$:
\begin{equation}\label{eq:X-alternative}
X=\sum_{i=1}^{+\infty} \val(\J_i)\ind_{\{\J_i\not\in\{\J_1,\dots,\J_{i-1}\}\}}\ind_{\{T_n\geq i\}}\, .
\end{equation}
As such, one may obtain concentration results for $X$ by resorting to the various methods designed for functions of independent variables.

The proof relies on the \emph{entropy method} as described in Chapter 6 of \citep{concentration}. We will show that $X$ is such that, for all $\lambda>0$,
\begin{equation}\label{eq:entropy}
\lambda\EE\left[X\e^{\lambda X}\right]-\EE\left[\e^{\lambda X}\right]\log\EE\left[\e^{\lambda X}\right] \leq \frac{\lambda^2 v}{2}\EE\left[\e^{\lambda X}\right]\, ,
\end{equation}
for $v$ as in \eqref{eq:def-variance}. Then, a classical argument due to Herbst (see \citep{concentration}, Proposition 6.1) ensures that, for all $\lambda>0$,
\begin{displaymath}
\log\EE\left[\e^{\lambda(X-\EE X)}\right]\leq \frac{\lambda^2 v}{2}\, ,
\end{displaymath}
and thus, for all $t>0$,
\begin{displaymath}
\PP\left(X-\EE X >t\right)\leq \exp\left(-\frac{t^2}{2v}\right)\, ,
\end{displaymath}
that is, the upper-tail of $X$ is sub-Gaussian with variance factor $v$. 
Let us establish inequality \eqref{eq:entropy}. For $t\geq 1$, consider the truncated variable $X_t$ defined by summing only from $1$ to $t$ in \eqref{eq:X-alternative}, i.e.
\begin{eqnarray*}
X_t&=&\sum_{i=1}^t \val(\J_i)\ind_{\{\J_i\not\in\{\J_1,\dots,\J_{i-1}\}\}}\ind_{\{T_n\geq i\}}\\
&:=&f(\J_1,\dots,\J_t)\, .
\end{eqnarray*}
Note that $X_t$ converges to $X$ almost surely as $t\to +\infty$.
Then, for all $1\leq i\leq t$, consider the perturbed variable $X_t^i$ which is obtained by replacing $\J_i$ by an independent copy $\J_i'$, i.e.
\begin{displaymath}
X_t^i=f(\J_1,\dots,\J_{i-1},\J_i',\J_{i+1},\dots,\J_t)\, ,
\end{displaymath}
and let $X^i$ be the almost sure limit of $X_t^i$, as $t\to +\infty$.
Theorem 6.15 of \citep{concentration} implies that, for all $\lambda>0$,
\begin{eqnarray}\label{eq:logSob_Xt}
\lambda\EE\left[X_t\e^{\lambda X_t}\right]-\EE\left[\e^{\lambda X_t}\right]\log\EE\left[\e^{\lambda X_t}\right]&\leq& \sum_{i=1}^t\EE\left[\lambda^2\e^{\lambda X_t}(X_t-X_t^i)_+^2\right]\, .
\end{eqnarray}
We now show that this inequality still holds when we let $t$ tend to $+\infty$. Let $\displaystyle{\val_{\text{max}}=\max_{1\leq j\leq N}\val(j)}$. For all $t\geq 1$, the variable $X_t$ is almost surely bounded by $n\val_{\text{max}}$. Hence, the left-hand side of \eqref{eq:logSob_Xt} tends to the left-hand side of \eqref{eq:entropy}. As for the right-hand side, we have that, for all $1\leq i\leq t$,
\begin{displaymath}
\EE\left[\lambda^2\e^{\lambda X_t}(X_t-X_t^i)_+^2\right]\leq \lambda^2\e^{\lambda n\val_{\text{max}}}\Delta^2\PP(i\leq T_n)\, ,
\end{displaymath}
and $\sum_{i=1}^{+\infty}\PP[i\leq T_n]=\EE[T_n]<+\infty$. Hence, by dominated convergence, the right-hand side also converges, and we obtain
\begin{eqnarray*}
\lambda\EE\left[X\e^{\lambda X}\right]-\EE\left[\e^{\lambda X}\right]\log\EE\left[\e^{\lambda X}\right]&\leq& \sum_{i=1}^{+\infty}\EE\left[\lambda^2\e^{\lambda X}(X-X^i)_+^2\right]\, .
\end{eqnarray*}
Recall that $(\I_1,\dots,\I_n)$ is the sequence of the first $n$ distinct items in $(\J_i)_{i\geq 1}$ and that $X$ is measurable with respect to $\sigma(\I_1,\dots,\I_n)$, so that
\begin{displaymath}
\sum_{i=1}^{+\infty}\EE\left[\lambda^2\e^{\lambda X}(X-X^i)_+^2\right]=\EE\left[\lambda^2\e^{\lambda X}\EE\left[\sum_{i=1}^{+\infty}(X-X^i)_+^2\Big| \I_1,\dots,\I_n\right]\right]\, .
\end{displaymath}
Thus, letting
\begin{displaymath}
V:=\EE\left[\sum_{i=1}^{+\infty}(X-X^i)_+^2\Big| \I_1,\dots,\I_n\right]\, ,
\end{displaymath}
our task comes down to showing that
\begin{displaymath}
V\leq \frac{v}{2}\quad\text{a.s.}\, .
\end{displaymath}
Observe that for all $i\geq 1$, we have $(X-X^i)_+^2\leq\Delta^2$ and that $X=X^i$ unless $i\leq T_n$ and one of the following two events occurs:
\begin{itemize}
\item $\J_i'\not\in\{\I_1,\dots,\I_n\}$;
\item the item $\J_i$ occurs only once before $T_{n+1}$.
\end{itemize}
Let us define 
\begin{eqnarray*}
A&=& \sum_{i=1}^{+\infty}\EE\left[\ind_{\{\J_i'\not\in\{\I_1,\dots,\I_n\}\}}\ind_{i\leq T_n}\Big| \I_1,\dots,\I_n\right]\, ,
\end{eqnarray*}
and 
\begin{eqnarray*}
B&=& \sum_{k=1}^{n}\EE\left[\ind_{\{\exists !\, i<T_{n+1},\, \J_i=\I_k\}}\Big| \I_1,\dots,\I_n\right]\, ,
\end{eqnarray*}
so that $V \leq\Delta^2\left(A+B\right)$. Since $\J_i'$ is independent of everything else and since $\sigma_n:=\s(\I_1)+\dots\s(\I_n)$ is a measurable function of $(\I_1,\dots,\I_n)$, we have
\begin{eqnarray*}
A&=& (1-\sigma_n)\EE\left[T_n\Big| \I_1,\dots,\I_n\right]\, .
\end{eqnarray*}
We use the following fact.
\begin{lemma}\label{lem:geom}
For $1\leq k\leq n$, let $\tau_k=T_k-T_{k-1}$. Conditionally on $(\I_1,\dots,\I_n)$, the variables $(\tau_k)_{k=1}^n$ are independent and for all $1\leq k\leq n$, $\tau_k$ is distributed as a Geometric random variables with parameters $1-\sigma_{k-1}$.
\end{lemma}
\begin{proof}
Let $(i_1,\dots,i_n)$ be an $n$-tuple of distinct elements of $\{1,\dots,N\}$ and let $t_1,\dots,t_n\geq 1$. Let also $(G_k)_{k=1}^n$ be independent Geometric random variables with parameter $(1-\s(i_1)-\dots -\s(i_{k-1}))$. We have
\begin{eqnarray*}
& &\PP\left((\tau_1,\dots,\tau_n)=(t_1,\dots,t_n),(\I_1,\dots,\I_n)=(i_1,\dots,i_n)\right)\\
& &\hspace{1cm}= \ind_{\{t_1=1\}} \s(i_1)\prod_{k=2}^n \left(\s(i_1)+\dots +\s(i_{k-1})\right)^{t_k-1}\s(i_k)\\
& &\hspace{1cm}=\prod_{k=1}^n\frac{\s({i_k})}{1-\s({i_1})-\cdots-\s({i_{k-1}})}\prod_{k=1}^n \PP\left(G_k=t_k\right)\\
& &\hspace{1cm}=\PP\left((\I_1,\dots,\I_n)=(i_1,\dots,i_n)\right)\prod_{k=1}^n \PP\left(G_k=t_k\right)\, ,
\end{eqnarray*}
and we obtain the desired result. 
\end{proof}
Lemma \ref{lem:geom} implies that
\begin{eqnarray*}
\EE\left[T_n\Big| \I_1,\dots,\I_n\right]&=& \sum_{k=1}^{n}\frac{1}{1-\sigma_{k-1}}\, .
\end{eqnarray*}
In particular, $A\leq n$. We also have
\begin{eqnarray}
\label{eq:bound-A}
A &\leq&  \frac{1}{\alpha}\sum_{k=1}^n \frac{N-n}{N-k+1} \,\leq\,  \frac{1}{\alpha}(N-n)\log\left(\frac{N}{N-n}\right)\, .
\end{eqnarray}
It remains to control $B$. Clearly $B\leq n$, which shows that $V\leq 2\Delta^2 n$. Moreover, for $1\leq k\leq n$, we have
\begin{eqnarray*}
\PP\left(\exists !\, i<T_{n+1},\, \J_i=\I_k \Big| \I_1,\dots,\I_n\right)&=&\EE\left[\prod_{j=k}^n\left(1-\frac{\s(\I_k)}{\sigma_j}\right)^{\tau_{j+1}-1}\Big| \I_1,\dots,\I_n\right] \, .
\end{eqnarray*}
Using Lemma \ref{lem:geom} and the fact that the generating function of a geometric variable $G$ with parameter $p$ is given by $\EE\left[x^{G}\right]=\frac{px}{1-(1-p)x}$, we obtain
\begin{eqnarray*}
B &=&  \sum_{k=1}^n \prod_{j=k}^n \frac{1}{1+\frac{\s(\I_k)}{1-\sigma_j}}\, .
\end{eqnarray*}
Thanks to the inequality the inequality $\log(1+x)\geq x-x^2/2$ for $x\geq 0$, 
\begin{eqnarray*}
B &\leq& \sum_{k=1}^n \prod_{j=k}^n \frac{1}{1+\frac{\alpha}{N-j}}\, \leq \, \sum_{k=1}^n  \exp\left(-\alpha\sum_{j=k}^n\frac{1}{N-j}+\frac{1}{2}\sum_{j=k}^n\frac{1}{(N-j)^2}\right)\, .
\end{eqnarray*}
The second term in the exponent is always smaller than $1/2$. Using Riemann sums, we get
\begin{eqnarray*}
B &\leq & 2 \sum_{k=1}^n \exp\left(-\alpha\log\left(\frac{N-k+1}{N-n}\right)\right)\,=\,2\sum_{k=1}^n\left(\frac{N-n}{N-k+1}\right)^{\alpha}\\
&\leq & \frac{2}{1-\alpha}N\left(\frac{N-n}{N}\right)^\alpha \, ,
\end{eqnarray*}
Combined with \eqref{eq:bound-A}, this yields
\begin{eqnarray*}
V&\leq & \left(\frac{1}{\alpha}\left(\frac{N-n}{N}\right)^{1-\alpha}\log\left(\frac{N}{N-n}\right)+\frac{2}{1-\alpha}\right)\Delta^2 N\left(\frac{N-n}{N}\right)^\alpha\\
&\leq & \left(\frac{\e^{-1}}{\alpha(1-\alpha)}+\frac{2}{1-\alpha}\right)\Delta^2 N\left(\frac{N-n}{N}\right)^\alpha\\
&\leq & \frac{1/2+2\alpha}{\alpha(1-\alpha)}\Delta^2N\left(\frac{N-n}{N}\right)^\alpha\, ,
\end{eqnarray*}
where the second inequality is due to the fact that $\log(x)/x^{1-\alpha}\leq \e^{-1}/(1-\alpha)$ for all $x>0$.

\hfill\qedsymbol

\section{Proof of Theorem \ref{thm:polya}}
\label{sec:proof-polya}

The proof of Theorem \ref{thm:polya} relies on the construction of a martingale coupling $(W,Z)$, \textit{i.e.} of a coupling of $W$ and $Z$ such that $\EE\left[Z\Big| W\right]=W$.

Consider two urns, $\bU_d$ and $\bU_D$, each of them initially containing $N$ balls, labelled from $1$ to $N$. In each urn, arrange the balls from left to right by increasing order of their label. Then arrange $\bU_D$ and $\bU_d$ on top of one another. Each time we will pick a ball in $\bU_D$, we will pick the ball just below it in $\bU_d$. More precisely, we perform an infinite sequence of steps as follows: at step $1$, we pick a ball $B_1$ uniformly at random in $\bU_D$ and pick the ball just below it in $\bU_d$. They necessarily have the same label, say $j$. We let $\K_1=\L_1=j$, and add, on the right part of $\bU_D$, $D-1$ balls with label $j$, and, on the right part of $\bU_d$, $d-1$ balls with label $j$ and $D-d$ unlabelled balls. Note that, at the end of this step, the two urns still have the same number of balls, $N+D-1$. The first step is depicted in Figure \ref{fig:dD_Polya}. Then, at each step $t$, we pick a ball $B_t$ at random among the $N+(t-1)(D-1)$ balls of $\bU_D$ and choose the ball just below it in $\bU_d$. There are two different possibilities:
\begin{itemize}
\item if the ball drawn in $\bU_d$ is unlabelled and the one drawn in $\bU_D$ has label $j$, we let $\L_t=j$ and add $D-1$ balls with label $j$ on the right part of $\bU_D$, and $D-1$ unlabelled balls on the right part of $\bU_d$, .
\item if both balls have label $j$, and if $t$ corresponds to the $i^{\text{th}}$ time a labelled ball is drawn in $\bU_d$, we let $\L_t=\K_i=j$ and add $D-1$ balls with label $j$ on the right part of $\bU_D$, and $d-1$ balls with label $j$ and $D-d$ unlabelled balls on the right part of $\bU_d$;
\end{itemize}

\begin{figure}[!htb]
\caption{The ball $B_1$ has label $2$ ($N=5$, $d=3$, $D=4$).}
\label{fig:dD_Polya}
\begin{center}
\begin{tikzpicture}[>=latex, scale=0.7]
\fill[brown] (1,3) circle (0.3cm);
\fill[brown] (2,3) circle (0.3cm);
\fill[brown] (3,3) circle (0.3cm);
\fill[brown] (4,3) circle (0.3cm);
\fill[brown] (5,3) circle (0.3cm);
\fill[brown] (6,3) circle (0.3cm);
\fill[brown] (7,3) circle (0.3cm);
\fill[brown] (8,3) circle (0.3cm);

\node at (1,3.6) {$\color{blue} \mathbf{1}$};
\node at (2,3.6) {$\color{blue} \mathbf{2}$};
\node at (3,3.6) {$\color{blue} \mathbf{3}$};
\node at (4,3.6) {$\color{blue} \mathbf{4}$};
\node at (5,3.6) {$\color{blue} \mathbf{5}$};
\node at (6,3.6) {$\color{blue} \mathbf{2}$};
\node at (7,3.6) {$\color{blue} \mathbf{2}$};
\node at (8,3.6) {$\color{blue} \mathbf{2}$};6

\fill[brown] (1,0) circle (0.3cm);
\fill[brown] (2,0) circle (0.3cm);
\fill[brown] (3,0) circle (0.3cm);
\fill[brown] (4,0) circle (0.3cm);
\fill[brown] (5,0) circle (0.3cm);
\fill[brown] (6,0) circle (0.3cm);
\fill[brown] (7,0) circle (0.3cm);
\draw (8,0) circle (0.3cm);

\node at (1,0.6) {$\color{blue} \mathbf{1}$};
\node at (2,0.6) {$\color{blue} \mathbf{2}$};
\node at (3,0.6) {$\color{blue} \mathbf{3}$};
\node at (4,0.6) {$\color{blue} \mathbf{4}$};
\node at (5,0.6) {$\color{blue} \mathbf{5}$};
\node at (6,0.6) {$\color{blue} \mathbf{2}$};
\node at (7,0.6) {$\color{blue} \mathbf{2}$};

\draw[->,black] (2,4.6) -- (2,4);
\node at (2,5) {$B_1$};

\draw[<->] (5.8,2.4) -- (8.2,2.4);
\node at (7,1.9) {$D-1$};
\draw[<->] (0.8,2.4) -- (5.2,2.4);
\node at (2.5,1.9) {$N$};
\draw[<->] (5.8,-0.6) -- (7.2,-0.6);
\node at (6.5,-1.1) {$d-1$};

\node at (-1,3) {$\bU_D$};
\node at (-1,0) {$\bU_d$};

\draw[<-] (8.4,0.5)--(9.2,1.5);
\node at (10,1.8) {unlabelled};
\end{tikzpicture}
\end{center}
\end{figure}
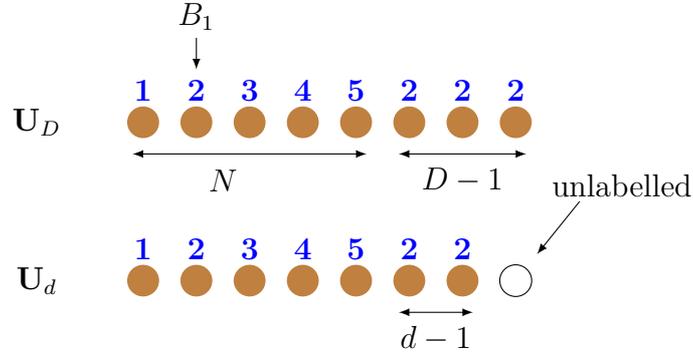

The sequence $(\K_1,\dots,\K_{n})$ records the labels of the first $n$ labelled balls picked in $\bU_d$, and $(\L_1,\dots,\L_{n})$ the labels of the first $n$ balls picked in $\bU_D$. Observe that $(\K_1,\dots,\K_{n})$ (resp. $(\L_1,\dots,\L_{n})$) is distributed as a $d$-Polya (resp. $D$-Polya) sample. Define
\begin{eqnarray*}
W&=&\val(\K_1)+\dots +\val(\K_n)\, ,\\
Z&=&\val(\L_1)+\dots +\val(\L_n)\, .
\end{eqnarray*}
Let us show that $1\leq i\leq n-1$, $\EE\left[\val(\L_{i+1})\Big| W\right]= \EE\left[\val(\L_i)\Big| W\right]$. Let $\{k_1,\dots_,k_n\}$ be a multiset of cardinality $n$ of elements of $\{1,\dots,N\}$, and let $\bA$ be the event $\{\K_1,\dots,\K_n\}=\{k_1,\dots_,k_n\}$ (accounting for the multiplicity of each label). Denote by $\cC_i$ the set of $D-1$ balls added at step $i$. Observe that, if $B_{i+1}\in\cC_i$, then $\L_{i+1}=\L_i$. Hence
\begin{eqnarray*}
\EE\left[\val(\L_{i+1})\Big| \bA\right]&=&\EE\left[\val(\L_i) \ind_{\{B_{i+1}\in\cC_i\}}\Big| \bA\right]+\EE\left[\val(\L_{i+1}) \ind_{\{B_{i+1}\not\in\cC_i\}}\Big| \bA\right]\, .
\end{eqnarray*}
We have
\begin{eqnarray*}
\EE\left[\val(\L_{i+1}) \ind_{\{B_{i+1}\not\in\cC_i\}}\Big| \bA\right]&=&\frac{1}{\PP(\bA)}\sum_{k=1}^N\val(k)\sum_{\ell=1}^N \PP\left(\L_i=\ell, \L_{i+1}=k, B_{i+1}\not\in\cC_i, \bA\right)\, .
\end{eqnarray*}
Notice that, on the event $B_{i+1}\not\in\cC_i$, the balls $B_i$ and $B_{i+1}$ are exchangeable. Hence $\PP\left(\L_i=\ell, \L_{i+1}=k, B_{i+1}\not\in\cC_i\right)=\PP\left(\L_i=k, \L_{i+1}=\ell, B_{i+1}\not\in\cC_i\right)$. Moreover, permuting $B_{i}$ and $B_{i+1}$ can not affect the multiset $\{\K_1,\dots,\K_n\}$. Hence
\begin{eqnarray*}
\EE\left[\val(\L_{i+1}) \ind_{\{B_{i+1}\not\in\cC_i\}}\Big| \bA\right]&=&\EE\left[\val(\L_{i}) \ind_{\{B_{i+1}\not\in\cC_i\}}\Big| \bA\right]\, ,
\end{eqnarray*}
and $\EE\left[\val(\L_{i+1})\Big| W\right]=\EE\left[\val(\L_{i})\Big| W\right]$. We get that, for all $1\leq i\leq n$,
\begin{eqnarray*}
\EE\left[\val(\L_i)\Big| W\right]&=& \EE\left[\val(\L_1)\Big| W\right]\,=\, \EE\left[\val(\K_1)\Big| W\right]\,=\,W/n\, ,
\end{eqnarray*}
where the last equality comes from the exchangeability of $(\K_1,\dots,\K_n)$.
\hfill\qedsymbol

\bibliographystyle{abbrv}
\bibliography{sampling}
\end{document}